\documentclass[12pt]{amsart}
\usepackage{amsmath}
\usepackage{amssymb}
\usepackage{amsthm}
\usepackage{amsfonts}

\numberwithin{equation}{section}
\newcommand{\h}{\mathcal{H}}
\newcommand{\C}{\mathbb{C}}
\newcommand{\Dta}{D_{\varphi}^{\theta,\alpha}}

\newcommand{\Ata}{A_{\varphi}^{\theta,\alpha}}
\newcommand{\Aat}{A_{\varphi}^{\alpha,\theta}}
\newcommand{\ka}{(K_{\alpha})^{\bot}}
\newcommand{\kt}{(K_{\theta})^{\bot}}
\newcommand{\apb}{AP^++BP^-}
\newcommand{\Dt}{D_{\varphi}^{\theta}}
\newcommand{\Daa}{D_{\bar\varphi}^{\alpha,\theta}}

\DeclareMathOperator{\sspan}{span}
\DeclareMathOperator{\Ima}{Im}
\DeclareMathOperator{\codim}{codim}
\DeclareMathOperator{\clos}{clos}
\DeclareMathOperator{\GCD}{GCD}
\DeclareMathOperator{\essinf}{ess\,inf}
\newtheorem{claim}{claim}[section]
\newtheorem{theorem}[claim]{Theorem}

\newtheorem{lemma}[claim]{Lemma}
\newtheorem{proposition}[claim]{Proposition}

\newtheorem{corollary}[claim]{Corollary}

\theoremstyle{definition}
\newtheorem{definition}[claim]{Definition}

\newtheorem{remark}[claim]{Remark}

\title[Dual truncated Toeplitz operators]{Invertibility, Fredholmness and kernels of dual truncated Toeplitz operators}
\thanks{The work of the first author was partially supported by FCT/Portugal through
UID/MAT/04459/2019. The research of the second and the fourth authors was financed by the Ministry of Science and Higher Education of the Republic of Poland}

\author[M. C. C\^amara ]{M. Cristina C\^amara}%

\address{M. Cristina C\^{a}mara
  Center for Mathematical Analysis, Geometry and Dynamical Systems, Mathematics Department, Instituto Superior T\'{e}cnico
 Universidade de Lisboa,  Av. Rovisco Pais, 1049-001 Lisboa, Portugal}
\email{cristina.camara@math.ist.utl.pt}

\author[K. Kli\'s--Garlicka]{Kamila Kli\'s--Garlicka}
\address{
	Kamila Kli\'s-Garlicka,  Department of Applied Mathematics,
  University of Agriculture,  ul. Balicka 253C, 30-198 Krak\'ow, Poland}
\email{rmklis@cyfronet.pl}

\author[ B. \L anucha ]{ Bartosz \L anucha }
\address{
	Bartosz {\L}anucha,  Institute of Mathematics, Maria Curie-Sk{\l}odowska University, pl. M.
	Curie-Sk{\l}odowskiej 1, 20-031 Lublin, Poland}
\email{bartosz.lanucha@poczta.umcs.lublin.pl}

\author[M. Ptak]{Marek Ptak}
\address{
	Marek Ptak, Department of Applied Mathematics,
 University of Agriculture, ul. Balicka 253C, 30-198 Krak\'ow, Poland}
\email{rmptak@cyf-kr.edu.pl}

\begin{document}

	\maketitle

\begin{abstract}{Asymmetric dual truncated Toeplitz operators acting between the orthogonal complements of two (eventually different) model spaces are introduced and studied. They are shown to be equivalent after extension to paired operators on $L^2(\mathbb T) \oplus L^2(\mathbb T)$ and, if their symbols are invertible in $L^\infty(\mathbb T)$, to asymmetric truncated Toeplitz operators with the inverse symbol. Relations with Carleson's corona theorem are also established. These results are used to study the Fredholmness, the invertibility and the spectra of various classes of dual truncated Toeplitz operators. }
	\end{abstract}
\setcounter{section}{-1}
\section{Introduction}

Toeplitz operators have been for a long time one of the most studied classes of nonselfadjoint operators (\cite {BS}). They are defined as compressions of multiplication operators on $L^2(\mathbb T)$, to the Hardy space of the unit disk $H^2(\mathbb D)$. Dual Toeplitz operators are analogously defined on the orthogonal complement of $H^2(\mathbb D)$, identified as usual with a subspace of $L^2(\mathbb T)$, as multiplication operators followed by projection onto $L^2(\mathbb T) \ominus H^2(\mathbb D)$. Although they differ in various ways from Toeplitz operators, they also share many properties, which is not surprising given that they are anti-unitarily equivalent. The algebraic and spectral properties of dual Toeplitz operators, and the extent to which their properties are parallel to those of Toeplitz operators on $H^2(\mathbb D)$, were studied in \cite{SZ}.

\vspace {2mm}

Truncated Toelitz operators, defined as compressions of multiplication operators to closed subspaces of $H^2(\mathbb D)$ which are invariant for the backward shift $S^*$, called model spaces, have also generated great interest, partly motivated by Sarason's paper \cite{Sa}.
Their study, as well as that of asymmetric truncated Toeplitz operators later introduced  in \cite{CP_ATTO}, raised many interesting questions and has led to new and sometimes surprising results, see for example \cite{BT,  CJKP, CP_ATTO, GMR}. It is natural to consider dual truncated Toeplitz operators, defined analogously as compressions of multiplication operators to the orthogonal complement of a model space in $L^2(\mathbb T)$. These operators were very recently introduced and studied in
 \cite{Ding Sang, Ding Qin Sang, Hu Dong}. It turns out that, in this case, they behave very differently from truncated Toeplitz operators. For instance, the symbol of a dual truncated Toeplitz operator is unique and the only compact operator of that kind is the zero operator, in sharp contrast with what happens with truncated Toeplitz operators on model spaces.

\vspace {2mm}

In this paper we study the kernels and various spectral properties, such as Fredholmness and invertibility, of dual truncated Toeplitz operators. The results are applied to describe the spectra of dual truncated Toeplitz operators in several classes including, as particular cases, the dual truncated shift and its adjoint. We do this by using a novel approach to dual truncated Toeplitz operators and their asymmetric analogues, defined similarly between the orthogonal complements of two possibly different model spaces. This involves proving their equivalence after extension to paired operators in $L^2(\mathbb T)\oplus L^2(\mathbb T)$, defined in Section 2, and establishing connections with the corona theorem. This allows moreover to show that, whenever their symbol is invertible in $L^\infty(\mathbb T)$, dual truncated Toeplitz operators are in fact equivalent after extension to truncated Toeplitz operators with the inverse symbol.

\vspace {2mm}

The paper is organized as follows. In Section 1 we introduce asymmetric dual truncated Toeplitz operators and present some basic properties, while in Section 2 we recall the concepts of paired operator and equivalence after extension between two Banach spaces. In Section 3 we study the solvability of certain equations involving asymmetric dual truncated Toeplitz operators in connection with equations involving paired operators. In Section 5 we show that dual truncated Toeplitz operators are equivalent after extension to truncated Toeplitz operators with the inverse symbol, if the latter is invertible in $L^\infty(\mathbb T)$. In Section 6 we study the kernels of asymmetric dual truncated Toeplitz operators in terms of explicitly defined isomorphisms with kernels of other operators. We show in particular that the kernels of a dual truncated Toeplitz operator and its adjoint are isomorphic and related by the usual conjugation on a model space. In Section 7 we present sufficient conditions for a dual truncated Toeplitz operator to be injective or invertible in terms of certain corona pairs, i.e., pairs of functions satisfying the hypotheses of Carleson's corona theorem (\cite{Sarason2, Garnett}). We use the previous results to study the Fredholmness, invertibility and spectra of several classes  of dual truncated Toeplitz operators.


\section{Elementary properties}

Let $P^+,\,P^-$ be the orthogonal projections from $L^2$ onto $H^2$ and $H^2_-=\bar z\overline {H^2}$, respectively. We have $P^-=I - P^+$.

Recall that for $\varphi\in L^\infty=L^\infty(\mathbb{T})$ a {\it Toeplitz operator} $T_\varphi\colon H^2\to H^2$ is defined by $T_\varphi f=P^+(\varphi f)$ for $f\in H^2$.

For an inner function $\theta$ define the {\it  model space} $K_\theta=H^2\ominus \theta H^2$ and let $P_\theta$ and $Q_\theta$ be the orthogonal projections from $L^2 =L^2(\mathbb T)$ onto the model space $K_\theta$ and its orthogonal complement $(K_{\theta})^{\bot}= L^2\ominus K_{\theta}=H^2_-\oplus \theta H^2$, respectively.

Let $\alpha,\theta$ be inner functions and let $\varphi \in L^2$. An {\it asymmetric truncated Toeplitz operator} $A^{\theta,\alpha}_\varphi$ is defined by
\begin{equation*}
  A_{\varphi}^{\theta,\alpha}f=P_{\alpha}\varphi P_{\theta} f,
\end{equation*}
for $f\in K_\theta\cap L^\infty$ (see \cite{CP_ATTO,CJKP}),
whereas the {\it asymmetric dual truncated Toeplitz operator} $D_{\varphi}^{\theta,\alpha}$ is defined by
\begin{equation*}
  D_{\varphi}^{\theta,\alpha}f=Q_{\alpha}\varphi Q_{\theta} f=P^-\varphi f+\alpha P^+\bar\alpha\varphi f
\end{equation*}
for $f\in (K_{\theta})^{\bot}\cap L^{\infty}$, which is a dense subset of $(K_{\theta})^{\bot}$. If  $A^{\theta,\alpha}_\varphi$ or $D_{\varphi}^{\theta,\alpha}$ have a bounded extension to $K_\theta$ or $(K_{\theta})^{\bot}$, respectively, we denote them also by $A^{\theta,\alpha}_\varphi$ and $D_{\varphi}^{\theta,\alpha}$, respectively. When $\alpha =\theta$ instead of $A^{\theta,\theta}_\varphi$ and $D_{\varphi}^{\theta,\theta}$ we write $A^{\theta}_\varphi$ and $D_{\varphi}^{\theta}$, respectively.

We start with some elementary properties of asymmetric dual truncated Toeplitz operators. These properties were proved in \cite{Ding Sang} for $\alpha=\theta$.

\begin{proposition}\label{1.1}
Let $\varphi \in L^2$. Then $\Dta$ is bounded if and only if $\varphi\in L^{\infty}$ and, in that case, $\|\Dta\|=\|\varphi\|_{\infty}$.
\end{proposition}

\begin{proof}
Let $\varphi \in L^2$ and take $f\in H^\infty$.  Then $\theta f \in \theta H^\infty \subset \theta H^2$,
$$\Dta(\theta f)=(P^-+\alpha P^+\bar\alpha)(\varphi \theta f)\in H^2_- \oplus \alpha H^2$$
and
\begin{displaymath}
\begin{split}
 \|\Dta(\theta f)\|^2&=\|P^-\varphi \theta f\|^2+\|\alpha P^+ \bar\alpha \varphi \theta f\|^2  \\
  &=\|P^-\varphi \theta f\|^2+\|P^+\bar\alpha \varphi\theta f\|^2 \geqslant \|T_{\bar\alpha\theta\varphi}(f)\|^2.
\end{split}
\end{displaymath}

If $\Dta$ is bounded, then, for some $C>0$, $\|\Dta(\theta f)\|^2 \leqslant C\|f\|^2$, so by the above $\|T_{\bar\alpha\theta\varphi}(f)\|^2\leqslant C\|f\|^2$.
Since this holds for any $f\in H^\infty$, it follows that $T_{\bar\alpha\theta\varphi}$ is bounded in $H^2$ and therefore $\bar\alpha\theta\varphi \in L^{\infty}$, which implies that $\varphi \in L^{\infty}$. Moreover, $\|\varphi\|_{\infty}=\|T_{\bar\alpha\theta\varphi}\|\leqslant \|\Dta\|$.

On the other hand, if $\varphi \in L^{\infty}$, then $\Dta$ is clearly a bounded operator from $(K_{\theta})^{\bot}$ into $(K_{\alpha})^{\bot}$. Indeed, for any $f\in \kt$ we then have
\[\|\Dta f\|=\|Q_{\alpha}\varphi f\|\leqslant \|\varphi f\|\leqslant \|\varphi\|_{\infty}\|f\|.\]
Moreover, $ \|\Dta\| \leqslant \|\varphi\|_{\infty}$.
\end{proof}

Taking the result of Proposition \ref{1.1} into account, we assume from now on that $\varphi \in L^\infty$. It is also easy to see that $(\Dta)^*=D_{\overline{\varphi}}^{\alpha,\theta}$.

\begin{proposition}\label{1.2}
For $\varphi \in L^{\infty}$, we have that $\Dta$ is compact if and only if $\varphi=0$.
\end{proposition}

\begin{proof}
Assume that $\Dta$ is compact and let $f_n\in\kt$, with $f_n$ weakly convergent to $0$ $(f_n\rightharpoonup 0)$. Then $\|\Dta f_n\|\to 0$. Note that for $f_n=\theta \tilde{f}_n$ with $\tilde{f}_n\in H^2$, we have $f_n\rightharpoonup 0 $
(in $\theta H^2$) if and only if $ \tilde f_n \rightharpoonup 0$ (in $H^2$). It follows that if $ \tilde f_n \rightharpoonup 0$, then $\|\Dta (\theta \tilde f_n)\|\to 0$. Since $\|T_{\bar\alpha\theta\varphi}\tilde f_n\|\leqslant \|\Dta (\theta \tilde f_n)\|$ (see the proof of Proposition \ref{1.1}), we have that $\|T_{\bar\alpha\theta\varphi} \tilde f_n\| \to 0$ whenever $ \tilde f_n \rightharpoonup 0$. So, if $\Dta$ is compact, then $T_{\bar\alpha\theta\varphi}$ is also compact and therefore $\varphi=0$.
\end{proof}

\begin{remark}\label{1.3}
Proposition \ref{1.2} implies, in particular, that the only symbol for the zero asymmetric dual truncated Toeplitz operator is $\varphi=0$. Since the question of $\Dta$ being the zero operator is equivalent to $\varphi$ being a multiplier from $\kt$ into $K_{\alpha}$, we conclude also that there are no non-trivial $L^\infty$-multipliers from $\kt$ into $K_\alpha$. In contrast with this, the question of whether there are non-trivial multipliers from $K_{\alpha}$ into $\kt$, which is equivalent to $\Aat$ being the zero operator, has a positive answer (\cite {CP_ATTO, CJKP}).
\end{remark}


\section{Paired operators and equivalence after extension}
For a Banach space $X$ denote by $\mathcal{L}(X)$ the space of all bounded linear operators $A\colon X\to X$. Let $P\in \mathcal{L}(X)$ be a projection and let $Q=I-P$ be its complementary projection. An operator of the form $AP+BQ$ or $PA+QB$, where $A,B \in \mathcal{L}(X)$, is called a \emph{paired operator} (\cite{MP}).

Paired operators are closely connected with operators of the form $PCP|_{\Ima P}$ and $QCQ|_{\Ima Q}$, where $C\in \mathcal{L}(X)$, which are called general Wiener-Hopf operators or operators of Wiener-Hopf type (\cite {MP, Speck WH}). To understand this relation, it will be useful to introduce here the concept of equivalence after extension for operators.


\begin{definition}[Equivalence after extension, \cite{Bart}]
Let $X$, $\tilde X$, $Y$, $\tilde Y$ be Banach spaces and let us use the term operator to mean a bounded linear operator.

The operators $T\colon X\to \tilde X$ and $S\colon Y\to \tilde Y$ are said to be (algebraically and topologically) \textit{equivalent} if and only if $T=ESF$ where $E$ and $F$ are invertible operators; in that case we use the notation $T \sim S$.

The operators $T$ and $S$ are \textit{equivalent after extension} ($T \overset{\star}\sim S$) if and only if there exist (possibly trivial) Banach spaces $X_0,Y_0$, called extension spaces, and invertible operators $E \colon \tilde Y \oplus Y_0 \to \tilde X \oplus X_0$, $F\colon X \oplus X_0 \to Y \oplus Y_0$, such that $$\begin{bmatrix} T & 0 \\ 0 & I_{X_0}\end{bmatrix} =E \begin{bmatrix} S & 0 \\ 0 & I_{Y_0}\end{bmatrix}F.$$
\end{definition}

Clearly, if  $T \sim S$, then $T \overset{\star}\sim S$.
Operators that are equivalent after extension share many properties. In particular we have the following.

\begin{theorem}[\cite{Bart}]\label{2.2} Let $T$ and $S$ be two operators, $T\colon X\to \tilde X$, $S\colon Y \to \tilde Y$, and assume that $T \overset{\star}\sim S$. Then
\begin{enumerate}
  \item $\ker T$ is isomorphic to $\ker S$, i.e., $\ker T \simeq \ker S$;
  \item $\Ima T$ is closed if and only if $\Ima S$ is closed and, in that case, $\tilde X/\Ima T \simeq \tilde Y / \Ima S$;
  \item  if one of the operators $T$, $S$ is generalized (left, right) invertible $($\cite{MP, Speck WH}$)$, then the other is generalized (left, right) invertible too;
  \item $T$ is Fredholm if and only if $S$ is Fredholm and, in that case, $\dim\ker T=\dim\ker S$, $\codim \Ima T=\codim \Ima S$.
\end{enumerate}\end{theorem}


It is not difficult to see that $PCP|_{\Ima P}$ (respectively, $QCQ|_{\Ima Q}$) is equivalent after extension to $CP+Q$ (respectively, $P+CQ$) and
\begin{equation}\label{e22}
  CP+Q \sim PC+Q \sim PCP+Q,
\end{equation}
because
$$CP+Q=(PCP+Q) (I+QCP)$$
and
$$PC+Q=(I+PCQ)(PCP+Q),$$
where
\begin{equation*}
  (I+PCQ)^{-1}=I-PCQ, \quad (I+QCP)^{-1}=I-QCP,
\end{equation*}
(and analogously $P+CQ \sim P+QC \sim P+QCQ$).

As an example of two operators which are equivalent after extension we have the following.

\begin{theorem}[\cite{CP_ATTO}]\label{2.3} Let $\varphi \in L^{\infty}$ and let $\alpha$, $\theta$ be inner functions. Then $\Ata \overset{\star}\sim  T_{\Phi}$ where $\Ata=P_{\alpha}\varphi P_{\theta}|_{K_{\theta}}$ and $T_{\Phi}$ is the block Toeplitz operator on $H^2\oplus H^2$ with symbol $\Phi=\begin{bmatrix}
                                   \bar\theta & 0 \\
                                   \varphi & \alpha
                                 \end{bmatrix}$.

\end{theorem}

Equivalence after extension for two operators $T$ and $S$ implies that there is a strong connection between the solvability of the equations $T\varphi=\psi$ and $Sx=y$, in particular as regards the existence and uniqueness of solutions.

In the next section we study the relations between the solutions of the equations
\begin{equation*}
  \Dta f= g, \quad f\in \kt, g\in \ka
\end{equation*}
and
\begin{equation*}
  (\apb)\Phi=\Psi, \quad \Phi,\Psi \in L^2\oplus L^2,
\end{equation*}
where $\apb$ is a certain paired operator on $L^2\oplus L^2$, as a first step to establishing the equivalence after extension of $\Dta$ to a paired operator $\apb$. Here $P^{\pm}(L^2\oplus L^2)=P^{\pm}L^2\oplus P^{\pm}L^2$.


\section{Solvability relations}\label{sec3}
Let $\varphi \in L^{\infty}$ and let $\alpha,\theta$ be inner functions. We define
\begin{equation*}
  A=\begin{bmatrix}
      \varphi \theta & -1 \\
      \varphi\theta\bar\alpha & \phantom{-}0
    \end{bmatrix},
  \quad  B=\begin{bmatrix}
      \varphi & \phantom{-}0 \\
    \bar\alpha\varphi & -1
    \end{bmatrix}.
\end{equation*}

\begin{theorem}\label{3.1}
\begin{enumerate}
  \item Let $f_-,g_-\in H^2_-$, $\tilde f_+, \tilde g_+\in H^2$. Then $$\Dta(f_-+\theta\tilde f_+)=g_-+\alpha\tilde g_+$$ implies that $$(\apb)\Phi=\Psi,$$
  where $\Phi,\Psi \in L^2\oplus L^2$ are given by
  \begin{align*}
    \Phi=&\begin{bmatrix}
           \phi_1 \\
           \phi_2
         \end{bmatrix}=
         \begin{bmatrix}
           f_-+\tilde f_+ \\
           (1+\bar\alpha)P_{\alpha}\varphi(f_-+\theta\tilde f_+)
         \end{bmatrix},\\
    \Psi=&\begin{bmatrix}
           \psi_1 \\
           \psi_2
         \end{bmatrix}=
         \begin{bmatrix}
           g_-+\alpha\tilde g_+ \\
           \bar\alpha g_-+\tilde g_+
         \end{bmatrix}.
  \end{align*}
  \item Let $\Phi,\Psi\in L^2\oplus L^2$, $
    \Phi=\begin{bmatrix}
           \phi_1 \\
           \phi_2
         \end{bmatrix},\
    \Psi=\begin{bmatrix}
           \psi_1 \\
           \psi_2
         \end{bmatrix}.$
   Then
$$(\apb)\Phi=\Psi$$
implies that $$\Dta(f_-+\theta\tilde f_+)=g_-+\alpha\tilde g_+,$$
where $f_-,g_-\in H^2_-$ and $\tilde f_+,\tilde g_+\in H^2$ are given by
\begin{align*}
  f_-=P^-\phi_1, & \quad g_-=P^-\psi_1,\\
   \tilde f_+=P^+\phi_1, & \quad \tilde g_+=P^+\psi_2.
\end{align*}
\end{enumerate}
\end{theorem}

\begin{proof}
To prove (1) assume that $f_-,g_-\in H^2_-$ and $\tilde f_+, \tilde g_+\in H^2$. Since by definition $\Dta (f_-+\theta \tilde f_+)=Q_{\alpha}\varphi (f_-+\theta \tilde f_+)$, the equality $\Dta (f_-+\theta \tilde f_+)=g_-+\alpha \tilde g_+$ means that there exists $\psi_+\in K_{\alpha}$ such that
$$\varphi(f_-+\theta \tilde f_+)=g_-+\alpha \tilde g_++\psi_+.$$
Equivalently, there exist $\psi_+\in H^2$ and $\psi_-\in H^2_-$ such that
\begin{displaymath}
\left\{ \begin{array}{rl}
	\varphi  f_-+\varphi \theta \tilde f_+  &= g_-+\alpha \tilde g_+ +\psi_+,\\
	\bar\alpha  \psi_+  &=\psi_-,
\end{array}\right.
\end{displaymath}
i.e.,
\begin{displaymath}
\left\{ \begin{array}{rl}
\varphi\theta \tilde f_+ -\psi_++ \varphi f_- & = g_-+\alpha \tilde g_+, \\
\theta\bar\alpha \varphi\tilde f_++\bar\alpha\varphi f_- - \psi_- & =\bar\alpha g_-+\tilde g_+.
\end{array}\right.
\end{displaymath}
This last system of equations can be written in matrix form as
$$\left(\begin{bmatrix}
	\varphi\theta & -1 \\
	\varphi\theta\bar\alpha & \phantom{-} 0
\end{bmatrix}P^++
\begin{bmatrix}
	\varphi & \phantom{-} 0 \\
	\bar\alpha\varphi & -1
\end{bmatrix}P^-\right)
\begin{bmatrix}
	\phi_1 \\
	\phi_2
\end{bmatrix} =  \begin{bmatrix}
	\psi_1 \\
	\psi_2
\end{bmatrix}$$
with $$\Psi=\begin{bmatrix}
	\psi_1 \\
	\psi_2
\end{bmatrix} =
\begin{bmatrix}
	g_-+\alpha\tilde g_+ \\
	\bar\alpha g_-+\tilde g_+
\end{bmatrix}$$
and
\[\Phi=\begin{bmatrix}
           \phi_1 \\
           \phi_2
         \end{bmatrix}=
         \begin{bmatrix}
           f_-+\tilde f_+ \\
           \psi_-+\psi_+
         \end{bmatrix}=
         \begin{bmatrix}
           f_-+\tilde f_+ \\
           (1+\bar\alpha)(\varphi f_-+\varphi\theta\tilde f_+-g_--\alpha \tilde g_+)
         \end{bmatrix}.\]
Now the result for $\Phi$ follows from the fact that
$$g_-  = P^-(\varphi f_-+\varphi\theta\tilde f_+)$$and
$$  \tilde g_+  =P^+\bar\alpha(\varphi f_-+\varphi\theta \tilde f_+).$$

To prove (2) let  $\Phi=\begin{bmatrix}
	\phi_1 \\
	\phi_2
\end{bmatrix}\in L^2\oplus L^2$, $\Psi=\begin{bmatrix}
	\psi_1 \\
	\psi_2
\end{bmatrix}\in L^2\oplus L^2$ and put $\phi_{i\pm}=P^{\pm}\phi_i$, $i=1,2$. Then $(\apb)\Phi =\Psi$ can be written as a system of equations
\begin{displaymath}
\left\{ \begin{array}{rl}
	\varphi\theta\phi_{1+}-\phi_{2+}+\varphi\phi_{1-}&=\psi_1,\\
	\theta\bar\alpha\varphi\phi_{1+}+\bar\alpha\varphi\phi_{1-}-\phi_{2-}&=\psi_2,
\end{array}\right.
\end{displaymath}
which is equivalent to
\begin{displaymath}
\left\{ \begin{array}{rl}
	\varphi\phi_{1-}+\varphi\theta\phi_{1+}&=\psi_1+\phi_{2+},\\
	\bar\alpha(\varphi\phi_{1-}+\varphi\theta\phi_{1+})&=\psi_2+\phi_{2-}.
\end{array}\right.
\end{displaymath}
Moreover, the above is equivalent to
\begin{displaymath}
\left\{ \begin{array}{rl}
	\varphi(\phi_{1-}+\theta\phi_{1+})&=\psi_1+\phi_{2+},\\
	\bar\alpha(\psi_1+\phi_{2+})&=\psi_2+\phi_{2-}.
\end{array}\right.
\end{displaymath}
The first equation in the system above implies that
\begin{displaymath}
\begin{split}
  \Dta(\phi_{1-}+\theta \phi_{1+})&=Q_{\alpha}\varphi(\phi_{1-}+\theta \phi_{1+})\\
  &=Q_{\alpha}(\psi_1+\phi_{2+})=P^-\psi_1+\alpha P^+\bar\alpha(\psi_1+\phi_{2+})
\end{split}
\end{displaymath}
and the second equation gives
\begin{equation*}
  \alpha P^+\bar\alpha(\psi_1+\phi_{2+})=\alpha P^+\psi_2,
\end{equation*}
that is,
$$\Dta(P^-\phi_{1}+\theta P^+\phi_{1})=P^-\psi_1+\alpha P^+\psi_2.$$
\end{proof}

The relations in Theorem \ref{3.1} imply that $\Dta$ and $\apb$ share many properties. Indeed, in the next section we show that the former is equivalent after extension to the latter.


\section{Equivalence after extension of $\Dta$ to a paired operator}

Let us introduce some notations (see \cite{CKP}). If $\mathcal{H}$, $\mathcal{K}_1$, $\mathcal{K}_2$ are Hilbert spaces and $A_1 \colon \mathcal{H}\to \mathcal{K}_1$, $A_2\colon \mathcal{H}\to \mathcal{K}_2$, $B_1\colon \mathcal{K}_1\to \mathcal{H}$, $B_2\colon \mathcal{K}_2 \to \mathcal{H}$ are bounded linear operators, we define
$$A_1 \diamond A_2 \colon \mathcal{H}\to \mathcal{K}_1 \oplus \mathcal{K}_2, \quad (A_1 \diamond A_2)h=A_1 h\oplus A_2h,$$
and
  $$B_1 \boxplus B_2\colon \mathcal{K}_1\oplus \mathcal{K}_2 \to \mathcal{H},  \quad (B_1 \boxplus B_2)(f+g)=B_1f+B_2g.$$


In what follows, $\alpha$ and $\theta$ are inner functions, $\varphi \in L^{\infty}$ and

\begin{equation}\label{w4.1}
  A=\begin{bmatrix}
      \varphi \theta & -1 \\
      \varphi\theta\bar\alpha & \phantom{-}0
    \end{bmatrix},
  \quad  B=\begin{bmatrix}
      \varphi & \phantom{-}0 \\
    \bar\alpha\varphi & -1
    \end{bmatrix},
\end{equation}
as at the beginning of Section \ref{sec3}.

\begin{proposition}\label{4.1}Let $\varphi \in L^{\infty}$ and let $\alpha$, $\theta$ be inner functions. Then
 $$\Dta \overset{\star}\sim Q_{\alpha}\varphi Q_{\theta} \boxplus P_{\alpha},$$ where $Q_{\alpha}\varphi Q_{\theta} \boxplus P_{\alpha}  \colon \kt \oplus K_{\alpha} \to L^2$.
\end{proposition}
\begin{proof}
An easy computation shows that
\begin{equation*}
  \begin{bmatrix}
    \Dta & 0 \\
    0 & I_{K_{\alpha}}
  \end{bmatrix}=E_1 \begin{bmatrix}
                      Q_{\alpha}\varphi Q_{\theta} \boxplus P_{\alpha} & 0 \\
                      0 & I_{\{0\}}
                    \end{bmatrix}F_1,
\end{equation*}
where $F_1 \colon \kt \oplus K_{\alpha}  \to (\kt \oplus K_{\alpha})\oplus \{0\}$ is defined for $f_-\in H^2_-$, $\tilde f_+\in H^2$ and $g_\alpha\in K_\alpha$  by
$$F_1((f_-+\theta\tilde f_+)\oplus g_{\alpha})= [(f_-+\theta\tilde f_+)\oplus g_{\alpha}]\oplus 0,$$
and $E_1 \colon L^2\oplus \{0\}   \to (K_\alpha)^\perp \oplus K_{\alpha}$ is defined for $f\in L^2$ by
  $$E_1(f\oplus 0)=Q_{\alpha}f\oplus P_{\alpha}f.$$
Clearly, $F_1$ and $E_1$ are invertible.
\end{proof}

\begin{theorem}\label{4.2}Let $\varphi \in L^{\infty}$ and let $\alpha$, $\theta$ be inner functions. Then
	\begin{equation*}
  \Dta \overset{\star}\sim \apb,
\end{equation*}
where $\apb \colon L^2\oplus L^2 \to L^2\oplus L^2$ is a paired operator with $A,B$ given by \eqref{w4.1}.
\end{theorem}

\begin{proof}
Given the result of Proposition \ref{4.1} and the fact that $\overset{\star}\sim$ is an equivalence relation and thus transitive, we only have to prove that
\begin{equation*}
 Q_{\alpha}\varphi Q_{\theta}\boxplus P_{\alpha} \overset{\star}\sim \apb. 
\end{equation*}
To this end, we note that $Q_{\alpha}\varphi Q_{\theta}\boxplus P_{\alpha} $ is obviously equivalent after extension to
\begin{equation*}
  \begin{bmatrix}
     Q_{\alpha}\varphi Q_{\theta}\boxplus P_{\alpha}  & 0 \\
    0 & I_{L^2}
  \end{bmatrix} \colon (\kt \oplus K_{\alpha})\oplus L^2 \to L^2\oplus L^2.
\end{equation*}
Using the relations from Section \ref{sec3} and rewriting them appropriately, we get (as can be verified independently)
\begin{equation*}
  \begin{bmatrix}
     Q_{\alpha}\varphi Q_{\theta}\boxplus P_{\alpha}  & 0 \\
    0 & I_{L^2}
  \end{bmatrix} =E
      \begin{bmatrix}
   \varphi P^-+ \varphi\theta P^+ & -P^+ \\
\bar\alpha\varphi P^-+\varphi\bar\alpha\theta P^+ & -P^-
  \end{bmatrix}F,
    \end{equation*}
where $F\colon (\kt \oplus K_{\alpha})\oplus L^2 \to L^2\oplus L^2$, $$F=\begin{bmatrix}
    (P^-+P^+\bar\theta)\boxplus 0 & 0 \\
    (P^-\varphi\bar\alpha+P^+\varphi)Q_{\theta}\boxplus (-P_{\alpha})  & -(P^-+\alpha P^+)
  \end{bmatrix},$$ and $E \colon L^2\oplus L^2 \to L^2\oplus L^2$, $$E=\begin{bmatrix}
    \alpha P^-\bar\alpha & \alpha P^+ \\
    P^+\bar\alpha & P^-
  \end{bmatrix}.$$ The operators $F$ and $E$ are invertible by Lemmas \ref{4.3} and \ref{4.4} below.
\end{proof}

The proofs of the following two lemmas are straightforward.
\begin{lemma}\label{4.3}
The operator $F\colon (\kt \oplus K_{\alpha})\oplus L^2 \to L^2\oplus L^2$,
\begin{equation*}
\quad F=\begin{bmatrix}
    (P^-+P^+\bar \theta)\boxplus 0 & 0 \\
    (P^+\varphi+P^-\varphi\bar\alpha)Q_{\theta}\boxplus (-P_{\alpha})  & -(P^-+\alpha P^+)
  \end{bmatrix},
 \end{equation*}
is invertible and  $ F^{-1}\colon L^2\oplus L^2  \to (\kt \oplus K_{\alpha})\oplus L^2$,
\begin{equation*}
\quad F^{-1}=\begin{bmatrix}
    (P^-+\theta P^+) \diamond P_{\alpha}\varphi (P^-+\theta P^+) & 0 \diamond (-P_{\alpha})\\
    \varphi\bar\alpha(P^-+\theta P^+) & -(P^-+ P^+\bar\alpha)
  \end{bmatrix}.
\end{equation*}
\end{lemma}

\begin{lemma}\label{4.4}
The operator $E\colon L^2\oplus L^2 \to L^2\oplus L^2$,
$$E=\begin{bmatrix}
\alpha P^-\bar\alpha & \alpha P^+ \\
P^+\bar\alpha & P^-
\end{bmatrix},$$
is invertible and  $E^{-1}=E$.
\end{lemma}

\begin{corollary}\label{4.5}
Let $\theta$ be an inner function. Then
 $$D_{\varphi}^{\theta}\overset{\star}\sim A_0P^++B_0P^-,$$ where
\begin{equation}\label{A0}
   A_0=\begin{bmatrix}
      \varphi \theta & -1 \\
      \varphi & \phantom{-}0
    \end{bmatrix},
  \quad  B_0=\begin{bmatrix}
      \varphi  & \phantom{-}0 \\
    \bar\theta\varphi & -1
    \end{bmatrix}.
\end{equation}
\end{corollary}

In what follows $\mathcal{G}L^{\infty}$ will denote the set of all invertible elements of the algebra $L^{\infty}$.

\begin{corollary}\label{4.6}
The operator $\Dta$ is semi-Fredholm (respectively Fredholm) if and only if $\apb$ is  semi-Fredholm (respectively Fredholm) on $L^2\oplus L^2$ and, in that case, we have $\varphi \in \mathcal{G}L^{\infty}$.
\end{corollary}

\begin{proof}
The equivalence is a consequence of Theorem \ref{2.2} and Theorem \ref{4.2}. To prove that $\varphi$ is invertible note that $|\det A|=|\det B|=|\varphi|$.  Since a necessary condition for the operator $\apb$ to be semi-Fedholm is that
\begin{equation*}
  \essinf\limits_{t\in \mathbb{T}}|\det A(t)|>0, \quad \essinf\limits_{t\in \mathbb{T}}|\det B(t)|>0
\end{equation*}
(\cite[Chapter V, Theorem 5.1]{MP}), we conclude that if $\Dta$ is semi-Fredholm, then $\varphi \in \mathcal{G}L^{\infty}$.
\end{proof}

In particular, we conclude that if $\Dt$ is invertible, then $\varphi \in \mathcal{G}L^{\infty}$ (\cite[Proposition 2.4]{Ding Sang}) and then, denoting by $\varphi (\mathbb T)$ the essential range of $\varphi \in L^\infty$, we have
\begin{equation}
\label{w4.3}
\varphi (\mathbb T)\subset \sigma_e(\Dt)\subset \sigma(\Dt)
\end{equation}
(see also \cite[Theorem 4.1]{Ding Qin Sang}).

\section{Equivalence relations between $\Dta$ and truncated Toeplitz operators}

In view of Corollary \ref{4.6}, the case where $\varphi$ is an invertible element of $L^{\infty}$ becomes particularly interesting. Assume then that
\begin{equation*}
  \varphi \in \mathcal{G}L^{\infty}.
\end{equation*}
This implies that, for $A,B$ defined by \eqref{w4.1}, we have $A, B\in \mathcal{G}(L^{\infty})^{2\times 2}$  with
  \begin{equation*}
   A^{-1}=\begin{bmatrix}
     \phantom{-}0 &\bar\theta\alpha\varphi^{-1} \\
     -1 &\alpha
    \end{bmatrix},
  \quad  B^{-1}=\begin{bmatrix}
      \varphi^{-1}  & \phantom{-}0 \\
    \bar\alpha& -1
    \end{bmatrix}.
  \end{equation*}
In that case
\begin{equation*}
  \apb=B(P^+B^{-1}AP^++P^-)(I+P^-B^{-1}AP^+),
\end{equation*}
where $B$ (identified with multiplication by $B$ on $L^2\oplus L^2$) is invertible, as well as $I+P^-B^{-1}AP^+$ because
\begin{equation*}
  (I+P^-B^{-1}AP^+)^{-1}=I-P^-B^{-1}AP^+.
\end{equation*}
Therefore $\apb$ is equivalent to $P^+CP^++P^-$, where
%
%
%
\begin{equation*}
   C=B^{-1}A=\begin{bmatrix}
     \theta & -\varphi^{-1} \\
     0 &-\bar\alpha
    \end{bmatrix}.
\end{equation*}
It follows from Theorem \ref{4.2} that
 $\Dta$ is equivalent after extension to $P^+CP^++P^-$.
It is easy to see that $P^+CP^++P^- \sim P^+GP^++P^-$, where
\begin{equation*}
   G=\begin{bmatrix}
     0 & 1 \\
     1 &0
    \end{bmatrix}C
    \begin{bmatrix}
     \phantom{-}0 & 1 \\
     -1 & 0
    \end{bmatrix}=\begin{bmatrix}
    \bar\alpha & 0 \\
    \varphi^{-1} &\theta
    \end{bmatrix}
    \end{equation*}
(since $G$ and $C$ differ by constant factors). Therefore we have that
\begin{equation*}
  \Dta \overset{\star}\sim P^+GP^++P^- \overset{\star}\sim P^+GP^+|_{H^2\oplus H^2}=T_G\overset{\star}\sim A^{\alpha,\theta}_{{\varphi}^{-1}},
\end{equation*}
where we took Theorem \ref{2.3} into account. We have thus proved the following:

\begin{theorem}\label{5.1}
Let $\alpha$, $\theta$ be inner functions.
If $\varphi \in \mathcal{G}L^{\infty}$, then $\Dta \overset{\star}\sim  A^{\alpha,\theta}_{{\varphi}^{-1}}$.
\end{theorem}

\begin{corollary}\label{5.2}
If $\varphi \in \mathcal{G}L^{\infty}$, then $\Dt \overset{\star}\sim  A^{\theta}_{{\varphi}^{-1}}$.
\end{corollary}

\begin{corollary}\label{5.3}
The operator $\Dt$ is Fredholm (respectively, invertible) if and only if $\varphi \in \mathcal{G}L^{\infty}$ and $ A^{\theta}_{{\varphi}^{-1}}$  is Fredholm (respectively, invertible).
\end{corollary}
\begin{proof}
If $\Dt$ is Fredholm, then by Corollary \ref{4.6} we have $\varphi \in \mathcal{G}L^{\infty}$ and since $\Dt \overset{\star}\sim  A^{\theta}_{{\varphi}^{-1}}$, it follows that $A^{\theta}_{{\varphi}^{-1}}$ is Fredholm. Conversely, if $\varphi \in \mathcal{G}L^{\infty}$ and $A^{\theta}_{{\varphi}^{-1}}$ is Fredholm, then Corollary \ref{5.2} implies that $\Dt$ is also Fredholm. The proof for invertibility is analogous.
\end{proof}

\section{Kernel Isomorphisms}
By Theorem \ref{2.2}, the kernels of two operators that are equivalent after extension, are isomorphic. Using the relations from Section 2, we describe here several of those isomorphisms. We use the same notation as in Section 4.

\begin{theorem}\label{6.1}
The map
 $$ \mathcal{N}  \colon \ker \Dta \to \ker(\apb),$$
$$  \mathcal{N}  (f_-+\theta \tilde f_+)=\begin{bmatrix}
                                         f_-+\tilde f_+ \\
                                         \varphi(1+\bar\alpha)(f_-+\theta\tilde f_+)
                                       \end{bmatrix},\quad\text{where}\  f_-\in H^2_-, \tilde f_+\in H^2,$$
is an isomorphism and
\[ \ker \Dta=(P^-+\theta P^+)\,P_1(\ker (AP^++BP^-)),\] where $P_1\left(\begin{bmatrix}x\\y\end{bmatrix}\right)=x$.
\end{theorem}

\begin{proof} Since
$$(1+\bar\alpha)P_\alpha \varphi(f_-+\theta \tilde f_+)= \varphi(1+\bar\alpha)(f_-+\theta\tilde f_+)-(1+\bar\alpha)\Dta(f_-+\theta\tilde f_+),$$
it follows from Theorem \ref{3.1} that $\Dta(f_-+\theta\tilde f_+)=0$ if and only if $$(\apb)\begin{bmatrix}
                                         f_-+\tilde f_+ \\
                                         \varphi(1+\bar\alpha)(f_-+\theta\tilde f_+)
                                       \end{bmatrix} = \begin{bmatrix}
                                                         0 \\
                                                         0
                                                       \end{bmatrix}.$$
Thus $\mathcal{N}$ is well defined and injective. To see that $\mathcal{N}$ is also surjective note that if
$$(\apb)\begin{bmatrix} \phi_1 \\ \phi_2
\end{bmatrix} = \begin{bmatrix}
0 \\
0
\end{bmatrix}, $$
then
\begin{displaymath}
\begin{split}
\phi_2&= \varphi \theta P^+\phi_{1}+\varphi P^-\phi_1+\varphi\theta\bar \alpha P^+\phi_1+\bar\alpha\varphi P^-\phi_1\\
&=\varphi(P^-\phi_1+\theta P^+ \phi_1)+\varphi\bar\alpha(P^-\phi_1+\theta P^+\phi_1)\\
&=\varphi(1+\bar \alpha)(P^-\phi_1+\theta P^+\phi_1),
\end{split}
\end{displaymath}
which means that $\phi_2$ is determined by $\phi_1$. Thus by Theorem \ref{3.1}(2), if $\Phi=\begin{bmatrix} \phi_1 \\ \phi_2
\end{bmatrix}$ is in $\ker( \apb)$,  then $\Phi=\mathcal{N}(P^-\phi_1+\theta P^+\phi_1)$.
\end{proof}

\begin{theorem}\label{6.2}
The map
 $$ \mathcal{N}_*  \colon \ker (\Dta)^* \to \ker(\apb)^* ,$$
  $$\mathcal{N}_*  (g_-+\alpha \tilde g_+)=\begin{bmatrix}
                                         g_- \\
                                         \tilde g_+
                                       \end{bmatrix},\quad\text{where}\ g_-\in H^2_-, \tilde g_+\in H^2,$$
is an isomorphism.
\end{theorem}

\begin{proof}
We have $(\Dta)^*=\Daa$ and $(\apb)^*=P^+A^*+P^-B^*$, where $A^*=\bar A^T$, $B^*=\bar B^T$ and we identify $A^*$ and $B^*$ with the corresponding multiplication operators on $L^2\oplus L^2$. Now, let $\Phi=\begin{bmatrix}
\phi_1 \\
\phi_2
\end{bmatrix}\in L^2\oplus L^2$. Then $(\apb)^*\Phi=0$ if and only if $P^+A^*\Phi=0$ and $P^-B^*\Phi=0$. The last two conditions are equivalent to
$$A^*\Phi \in (H^2_-\oplus H^2_-)\quad\text{and}\quad B^*\Phi\in (H^2\oplus H^2),$$
that is,
\begin{displaymath}
\begin{bmatrix}
\bar\varphi\bar\theta & \bar\varphi\bar\theta\alpha \\
-1 & 0
\end{bmatrix}
\begin{bmatrix}
\phi_1 \\
\phi_2
\end{bmatrix}
\in (H^2_-\oplus H^2_-),\text{ and }
\begin{bmatrix}
\bar\varphi & \alpha \bar\varphi\\
0&-1
\end{bmatrix}
\begin{bmatrix}
\phi_1 \\
\phi_2
\end{bmatrix}
\in (H^2\oplus H^2).
\end{displaymath}
In other words, $\phi_1\in H^2_-$, $\phi_2\in H^2$ and
\begin{displaymath}
\left\{\begin{array}{c}
	\bar\theta (\bar\varphi\phi_1+\alpha\bar\varphi\phi_2)
	\in H^2_-, \\
	\bar\varphi\phi_1+\alpha\bar\varphi\phi_2
	\in H^2.
\end{array}\right.
\end{displaymath}
The two conditions above can be written as $\bar\varphi(\phi_1+\alpha\phi_2)\in K_{\theta}$. Therefore,  $(\apb)^*\Phi=0$ if and only if $\phi_1\in H^2_-$, $\phi_2\in H^2$ and $\Daa(\phi_1+\alpha\phi_2)=0$. This implies that $\mathcal{N}_*$ is well defined and surjective.
It is now easy to see that $\mathcal{N}_*$ is  injective.
\end{proof}


A {\it conjugation} on a Hilbert space $\h$ is an antilinear isometric involution  (see for instance \cite{GP}).
In what follows let $C_{\theta}\colon L^2\to L^2$ denote the conjugation defined as $$C_\theta(f)=\theta \bar z \bar f\qquad\text{ for}\ \  f\in L^2.$$
The conjugation $C_\theta $ preserves both
 the model space $K_\theta$ and its orthogonal complement $(K_\theta)^\perp$, (i.e., $C_\theta P_\theta=P_\theta C_\theta$), and therefore induces a conjugation in $K_\theta$ and in $(K_\theta)^\perp$, which we also denote by $C_\theta$. This conjugation plays an important role in the study of truncated Toeplitz operators.

\begin{theorem}\label{6.3}
The map
 $$ \mathcal{N}_D \colon \ker D^{\theta}_{\varphi} \to \ker(D^{\theta}_{\varphi})^*,$$
$$  \mathcal{N}_D (f_-+\theta \tilde f_+)=\overline{z\tilde f_+}+\theta \overline{z f_-},\quad\text{where}\ \  f_-\in H_-^2, \tilde f_+\in H^2,$$
is an isomorphism and  \[\ker(\Dt)^*=C_\theta (\ker \Dt).\]
\end{theorem}
\begin{proof}
	Let $f_-\in H_-^2$ and $\tilde f_+\in H^2$. Then $D^{\theta}_{\varphi}(f_-+\theta\tilde f_+)=0$ if and only if $\varphi(f_-+\theta\tilde f_+)\in K_{\theta}$, that is, if and only if
$$\theta \bar z \overline{\varphi(f_-+\theta\tilde f_+)}=\bar\varphi\,(\overline{z\tilde f_+}+\theta \,\overline{z f_-})\in K_\theta.$$
   The above means that
   $$D_{\bar\varphi}^{\theta}(g_-+\theta\tilde g_+)=0$$
   with $g_-=\overline{z\tilde f_+}\in H^2_-$, $\tilde g_+=\overline{z f_-}\in H^2 $ (i.e., $g_-+\theta\tilde g_+=\mathcal{N}_D(f_-+\theta \tilde f_+)$). Thus $\mathcal{N}_D$ is an isomorphism and we have $$\ker(\Dt)^*=\theta \bar z \overline{\ker\Dt}=C_\theta (\ker \Dt).$$
\end{proof}


\begin{corollary}\label{6.4}
If $D^{\theta}_{\varphi}$ is Fredholm, then it has Fredholm index $0$.
\end{corollary}

\begin{corollary}\label{6.5}
If $D^{\theta}_{\varphi}$ is Fredholm, then it is invertible if and only if $\ker D^{\theta}_{\varphi} =\{0\}$.
\end{corollary}

\begin{theorem}\label{6.6}
If $\varphi \in \mathcal{G}L^{\infty}$, then the map
 $$ \mathcal{N}_{DA} \colon \ker D^{\theta,\alpha}_{\varphi} \to \ker A^{\alpha,\theta}_{\varphi^ {-1}},$$
$$\mathcal{N}_{DA} (f_-+\theta \tilde f_+)=\varphi(f_-+\theta \tilde f_+),\quad \text{where}\ \  f_-\in H_-^2, \tilde f_+\in H^2,$$
is an isomorphism and  \[\ker  \Dta = \varphi^{-1}\ker A^{\alpha,\theta}_{\varphi^{-1}}.\]
\end{theorem}
\begin{proof}
Let $f\in L^2$. Then $f\in \ker \Dta$ if and only if $f \in \kt$ and $\varphi f \in K_\alpha$. In other words, $g=\varphi f\in K_{\alpha}$ and $\varphi^{-1}g=f\in  \kt$, that is, $g\in \ker A^{\alpha,\theta}_{\varphi^ {-1}}$.
\end{proof}


\section{Dual truncated Toeplitz operators and the corona theorem}
 Corona problems, seen as left invertibility problems, have a strong connection with the invertibility and Fredholmness of block Toeplitz operators (see for instance \cite {CDR} and references in it). In this section we extend some of those connections to paired operators and apply them to the study of injectivity and invertibility of dual truncated Toeplitz operators.

Let $CP^{\pm}$ denote the sets of corona pairs, i.e., pairs of functions satisfying the so called corona conditions in $\mathbb{D}$ and $\C\setminus(\mathbb{D}\cup \mathbb{T})$, denoted here by  $\mathbb{D}^\pm$, respectively:
\begin{equation*}
CP^{+}=\left\{\mathfrak{h}_+=\begin{bmatrix}
h_{1+}\\
h_{2+}
\end{bmatrix}\in (H^{\infty}\oplus H^{\infty}): \inf\limits_{z\in \mathbb{D}^+}(|h_{1+}(z)|+|h_{2+}(z)|)>0\right\}.
\end{equation*}
\begin{equation*}
CP^{-}=\left\{\mathfrak{h}_-=\begin{bmatrix}
h_{1-}\\
h_{2-}
\end{bmatrix}\in (\overline {H^{\infty}}\oplus \overline {H^{\infty}}): \inf\limits_{z\in \mathbb{D}^{-}}(|h_{1-}(z)|+|h_{2-}(z)|)>0\right\}.
\end{equation*}
Obviously,  $\mathfrak{h}_-\in CP^-$ if and only if $\overline {\mathfrak{h}_-}\in CP^+$.


By the corona theorem $\mathfrak{h}_{\pm}\in CP^{\pm}$ if and only if there exist $\tilde{\mathfrak{h}}_{\pm}$ with $\tilde{\mathfrak{h}}_+ \in (H^{\infty}\oplus H^{\infty})$ and $\tilde{\mathfrak{h}}_- \in \overline {H^{\infty}}\oplus \overline {H^{\infty}}$ such that
\begin{equation}\label{w8.2}
  (\tilde{\mathfrak{h}}_{\pm})^T \mathfrak{h}_{\pm}=1 \quad \text{in} \quad \mathbb{D}^{\pm}.
\end{equation}

\begin{theorem}\label{8.1}
Let $A,B\in (L^{\infty})^{2\times 2}$ be such that $\det A=\varphi f_+$ and $\det B=\varphi f_-$, where $\varphi \in \mathcal{G}L^{\infty}$ and $f_+,\overline{f_-}\in H^{\infty}\setminus \{0\}$. If there exist $\mathfrak{h}_{\pm}\in CP^{\pm}$ satisfying
\begin{equation}\label{w8.3}
  A\mathfrak{h}_++B\mathfrak{h}_-=0
\end{equation}
with $A\mathfrak{h}_+(t)\neq 0$ a. e. on $\mathbb{T}$, then the operator $\apb$ is injective in $L^2\oplus L^2$.
\end{theorem}

\begin{proof}
Let $\Phi=\begin{bmatrix}
\phi_1\\
\phi_2
\end{bmatrix}\in \ker(\apb)\subset L^2\oplus L^2$. Taking $\Phi_{\pm}=P^{\pm}\Phi$ 
 we have
\begin{equation}\label{w8.4}
  A\Phi_+=-B\Phi_-.
\end{equation}
So, using \eqref{w8.3}, we can write
\begin{equation*}
  A\begin{bmatrix} \mathfrak{h}_+ & \Phi_+\end{bmatrix}=-B\begin{bmatrix}\mathfrak{h}_- & \Phi_-\end{bmatrix}.
\end{equation*}
Taking determinants on both sides, we get
 $$ \varphi f_+\det\begin{bmatrix} \mathfrak{h}_+ & \Phi_+\end{bmatrix}  = -\varphi f_- \det\begin{bmatrix}\mathfrak{h}_- & \Phi_-\end{bmatrix},$$
 which, since $\varphi \in \mathcal{G}L^{\infty}$, is equivalent to
$$f_+\det\begin{bmatrix} \mathfrak{h}_+ & \Phi_+\end{bmatrix}  = -f_- \det\begin{bmatrix}\mathfrak{h}_- & \Phi_-\end{bmatrix}.$$
Moreover, since the left hand side of the above equality represents a function in $H^2$ while the right hand side represents  a function in $H^2_-$, both sides must be zero. It follows that
\begin{equation*}
    \det\begin{bmatrix}\mathfrak{h}_{\pm} & \Phi_{\pm}\end{bmatrix}=0,
\end{equation*}
so there exist  non-zero linear combinations, with $\delta_{1\pm}$ and $\delta_{2\pm}$ defined a. e. on $\mathbb{T}$,
\begin{equation*}
    \delta_{1\pm}\mathfrak{h}_{\pm} + \delta_{2\pm} \Phi_{\pm}=0.
\end{equation*}
Multiplying on the left by $(\tilde{\mathfrak{h}}_{\pm})^T$, by \eqref{w8.2} we have
\begin{equation*}
  \delta_{1\pm}=-\delta_{2\pm}(\tilde{\mathfrak{h}}_{\pm})^T \Phi_{\pm}.
\end{equation*}
Therefore there exist $\delta_{\pm}\neq 0$ a. e. on $\mathbb{T}$ such that
\begin{equation}\label{eq84}
\Phi_{\pm}=\delta_\pm \mathfrak{h}_\pm \quad\text{on } \mathbb{T}.
\end{equation}
From \eqref{w8.3} and \eqref{eq84},
 $ A\Phi_+=-B\Phi_-$ if and only if $\delta_+A\mathfrak{h}_+=-\delta_-B\mathfrak{h}_-$, or $(\delta_+-\delta_-)A\mathfrak{h}_+=0$. This however happens if and only if $\delta_+=\delta_-$ (as $A\mathfrak{h}_+(t)\neq 0$ a. e. on $\mathbb{T}$).
Since $\delta_{+}=(\tilde{\mathfrak{h}}_{+})^T\Phi_{+}\in H^2$, $\delta_{-}=(\tilde{\mathfrak{h}}_{-})^T\Phi_{-}\in H^2_{-}$ and $\delta_+=\delta_-$, we have $\delta_{\pm}=0$ and therefore $\Phi_{\pm}=0$.
\end{proof}

\begin{theorem}\label{8.2}
Let $A,B\in (L^{\infty})^{2\times 2}$ be invertible matrices such that 
$\det (B^{-1}A)=f_- z^k f_+$ with $f_+, \overline{f_-}\in \mathcal{G}H^\infty
$, $k\in \mathbb{Z}$. 
If there exist $\mathfrak{h}_{\pm}\in CP^{\pm}$ satisfying \begin{equation}\label{w8.3a}
 A\mathfrak{h}_++B\mathfrak{h}_-=0,
\end{equation} then $\apb$ is invertible, injective (and not surjective), or surjective (and not injective) if and only if $k=0$, $k>0$ or $k<0$, respectively.
\end{theorem}

\begin{proof}
If $\mathfrak{h}_{\pm}$ satisfy \eqref{w8.3a}, then $B^{-1}A\mathfrak{h}_+=-\mathfrak{h}_-$, where $\det(B^{-1}A)=f_- z^k  f_+$ with $f_+, \overline{f_-}\in \mathcal{G}H^\infty
$, $k\in \mathbb{Z}$ and $\mathfrak{h}_{\pm}\in CP^{\pm}$. Therefore, by \cite[Theorem 4.5]{CDR} and the proof given there, we have that the Toeplitz operator $T_{B^{-1}A}\colon (H^2\oplus H^2) \to (H^2\oplus H^2)$ is invertible, only injective, or only surjective if and only if $k=0$, $k>0$ or $k<0$, respectively. 
We conclude that the same holds for $\apb$ because (by \eqref{e22}),
\begin{equation*}
  T_{B^{-1}A}\overset{\star}\sim  B^{-1}AP^++P^- \sim \apb.
\end{equation*}
\end{proof}

Taking Theorem \ref{4.2} into account, we also have the following.

\begin{theorem}\label{8.3}
Let $\theta,\alpha$ be inner functions such that $\alpha \leqslant \theta$ (i.e., $\alpha$ divides $\theta$).
\begin{enumerate}
  \item If $\varphi \in \mathcal{G}L^{\infty}$, and there exist $h_{1+},h_{2+}, \overline{h_{2-}}\in H^\infty$ satisfying
      \begin{equation}\label{w8.5}
         \begin{bmatrix}
         h_{1+}\\
         h_{2+}
         \end{bmatrix} \in CP^+, \quad  \begin{bmatrix}
         \bar\alpha h_{1+}\\
         h_{2-}
         \end{bmatrix}\in CP^-
         \end{equation}
         and
         \begin{equation}\label{w8.6}
          \varphi (h_{2-}+\theta h_{2+})=h_{1+},
      \end{equation}
      then $\ker\Dta=\{0\}$.
  \item If, moreover, $\theta\bar\alpha$ is a finite Blaschke product of degree $k$, then $\Dta$ is injective; it is invertible if and only if $\alpha=c\theta$ with $c\in \mathbb{C}, |c|=1$; so in particular $\Dt$ is invertible.
\end{enumerate}
\end{theorem}
\begin{proof}
To prove (1) note that for $A$, $B$ given by \eqref{4.1}, we have $$\det A=\varphi\theta\bar\alpha,\quad \det B=-\varphi,$$ where $\theta\bar\alpha\in H^{\infty}$. For $h_{1-}=\bar\alpha h_{1+}\in \overline{H^{\infty}}$, we have from \eqref{w8.6} 
that
\begin{equation*}
\left\{  \begin{array}{c}
     \varphi\theta h_{2+}+\varphi h_{2-}-h_{1+}=0,\\
    \bar\alpha\varphi\theta h_{2+}+\bar\alpha\varphi h_{2-}-h_{1-}=0,
  \end{array}\right.
\end{equation*}
which can be written as
  \begin{equation*}
  A\begin{bmatrix}
                         h_{2+} \\
                         h_{1+}
                       \end{bmatrix}+
                       B\begin{bmatrix}
                         h_{2-} \\
                         h_{1-}
                       \end{bmatrix}=0.
\end{equation*}
By \eqref{w8.5},  $\begin{bmatrix}
	h_{2+}\\
	h_{1+}
\end{bmatrix} \in CP^+$ and $\begin{bmatrix}
	 h_{2-}\\
	h_{1-}
\end{bmatrix}\in CP^-$, so Theorem \ref{8.1} implies that $\apb$ is injective. By Theorems \ref{2.2} and \ref{4.2} the operator $\Dta$ is also injective.

Now we prove (2). If $\theta\bar\alpha$ is a finite Blaschke product of degree $k$, then we can factorize $\theta\bar\alpha=R_- z^k R_+$, where $R_{\pm}$ are rational functions such that $R_+,\bar{R}_-\in \mathcal{G}H^{\infty}$. Since $\det (B^{-1}A)\,=\,-\theta\bar\alpha$, the result follows from Theorems \ref{8.2} and \ref{4.2}.
\end{proof}

Dual truncated Toeplitz operators can also be related to corona problems by using Corollary \ref{5.2} and the known relations between truncated Toeplitz operators and the corona theorem (\cite {CP_ATTO, CP_spectral}), as in the following theorem which will be used in the next section to describe the spectrum of a class of dual truncated Toeplitz operators with analytic symbols.

\begin{theorem}\label{8.4}
Let $\theta$ be an inner function. If $\varphi \in \mathcal{G}L^{\infty}$ and there exist $\mathfrak{h}_+ \in H^\infty\oplus H^\infty,\,\overline{\mathfrak{h}_-}\in CP^+$ such that $G \mathfrak{h}_+=\mathfrak{h}_-$ with
\begin{equation*}
   G=\begin{bmatrix}
     \bar\theta & 0 \\
     \varphi^{-1} &\theta
    \end{bmatrix},
\end{equation*}
then $\Dt$ is invertible if and only if $\mathfrak{h}_+\in CP^+$.
\end{theorem}

\begin{proof}
In this case $\Dt$ is invertible if and only if $A^{\theta}_{\varphi^{-1}}$ is invertible, and this is equivalent to the Toeplitz operator $T_G$ being invertible. The result now follows from Theorem 3.11 in \cite{CP_spectral}.
\end{proof}

\section{Fredholmness, inwertibility and spectra of dual truncated Toeplitz operators}

In this section we apply the previous results to study several classes of dual truncated Toeplitz operators. In what follows $\theta$ is always an inner function.

\subsection{Analytic symbols}
We start by using the results of Theorem \ref{8.4} to describe the spectrum of dual truncated Toeplitz operators with analytic symbols of a particular type.

\begin{theorem}\label{9.1}
	If $\varphi \in H^{\infty}$ and $\bar\theta \varphi\in \overline{H^{\infty}}$, then $\sigma(D_{\varphi}^{\theta})=\clos\varphi(\mathbb{D})$.
\end{theorem}

\begin{proof}
	If $\lambda \in \varphi(\mathbb{T})$, then $\varphi - \lambda \notin \mathcal{G}L^{\infty}$, so $D_{\varphi-\lambda}^{\theta}$ is not Fredholm. If $\lambda \notin \varphi(\mathbb{T})$, then $\varphi-\lambda \in \mathcal{G}L^{\infty}$ and we have
	\[\bmatrix \bar\theta & 0 \\ \frac 1{\varphi-\lambda} & \theta \endbmatrix
	\bmatrix \varphi-\lambda \\ 0 \endbmatrix=
	\bmatrix \bar\theta(\varphi-\lambda) \\ 1\endbmatrix.\]
	Since the right hand side of the above equality belongs to $CP^-$, by Theorem \ref{8.4} we have that $D_{\varphi-\lambda}^{\theta}$ in invertible if and only if $\begin{bmatrix}
	\varphi-\lambda \\
	0
	\end{bmatrix}\in CP^+$. The latter is equivalent to $(\varphi-\lambda)\in \mathcal{G} H^{\infty}$, i.e., $$\inf_{z\in \mathbb{D}}|\varphi(z)-\lambda|>0.$$ Since $\varphi(\mathbb{T})\subset \clos\varphi(\mathbb{D})$, we conclude that $\sigma(D_{\varphi}^{\theta})=\clos\varphi(\mathbb{D}) $.
\end{proof}

The assumptions of Theorem \ref{9.1}, are satisfied in particular if $\varphi \in K_{z\theta} \cap L^{\infty}$ since, in that case, $\varphi \in H^{\infty}$ and $\bar\theta \varphi \in \overline{H^{2}}\cap L^{\infty}=\overline{H^{\infty}}$ (see \cite[Theorem 4.3]{Ding Qin Sang}).

Several other results regarding analytic symbols will be obtained from the properties studied below.

\subsection{Symbols with analytic or co-analytic inverse}

We consider now symbols $\varphi\in \mathcal{G}L^{\infty}$ such that $\varphi^{-1}$ is in $H^{\infty}$ or in $\overline{H^{\infty}}$. In that case, by Corollary \ref{5.2}, $D_{\varphi}^{\theta}$ is equivalent after extension to a truncated Toeplitz operator $A^{\theta}_{\varphi^{-1}}$; therefore we start by recalling the following.

\begin{theorem}[\cite{CP_ATTO, CP_spectral}]\label{9.2}
	Let $g_+\in H^{\infty}$ and denote by $g_+^i$ the inner factor of $g_+$. Then:
	\begin{enumerate}
		\item $A^{\theta}_{g_+}$ is Fredholm if and only if $\gamma=\GCD(\theta,g_+^i)$ is a finite Blaschke product and $\bar\gamma \begin{bmatrix}
		\theta \\
		g_+
		\end{bmatrix}\in CP^+$,
		\item $A^{\theta}_{g_+}$ is invertible if and only if $\begin{bmatrix}
		\theta \\
		g_+
		\end{bmatrix}\in CP^+$,
		\item $\ker A^{\theta}_{g_+} = \frac{\theta}{\gamma}K_{\gamma}$.
	\end{enumerate}
\end{theorem}

Using Theorem \ref{2.2} and Corollary \ref{5.2} we now obtain the following.

\begin{theorem}\label{9.3}
	Let $\varphi \in \mathcal{G}L^{\infty}$. Assume that $\varphi^{-1}\in H^{\infty}$ and let $\varphi^{-1}=\beta a_+$ be its standard inner--outer decomposition ($\beta$--inner, $a_+$--outer). Denote $\gamma = \GCD(\theta,\beta)$. Then
	\begin{enumerate}
		\item $D_{\varphi}^{\theta}$ is Fredholm if and only if $\gamma$ is a finite Blaschke product,
		\item $D_{\varphi}^{\theta}$ is invertible if and only if $\begin{bmatrix}
		\theta \\
		\beta
		\end{bmatrix}\in CP^+$,
		\item $\ker D_{\varphi}^{\theta}=\beta a_+ \frac{\theta}{\gamma} K_{\gamma}\subset \frac{\theta\beta}{\gamma}H^2$.
	\end{enumerate}
\end{theorem}

\begin{proof}
	Parts (1) and (2) follow from Theorem \ref{9.2}, Theorem \ref{2.2} and Corollary \ref{5.2}. Note that the condition $\begin{bmatrix}
	\theta \\
	\beta
	\end{bmatrix}\in CP^+$ in (2) is equivalent to $\begin{bmatrix}
	\theta \\
	\varphi^{-1}
	\end{bmatrix}\in CP^+$ in this case. On the other hand, by Theorem \ref{6.6},
	\[\ker D_{\varphi}^{\theta}=\varphi^{-1}\ker A_{\varphi^{-1}}^{\theta}=\beta a_+\tfrac{\theta}{\gamma}K_{\gamma}.\]
\end{proof}

\begin{corollary}\label{9.4}
	If $\varphi\in \mathcal{G}H^{\infty}$, then $D_{\varphi}^{\theta}$ is invertible.
\end{corollary}

\begin{proof}
	In this case $\beta \in \mathbb{C}$ with $|\beta|=1$, therefore $\begin{bmatrix}
	\theta \\
	\beta
	\end{bmatrix}\in CP^+$.
\end{proof}

\begin{corollary}\label{9.5}
	If $\theta$ is a singular inner function, then $D^{\theta}_{\varphi}$ is Fredholm if and only if it is invertible.
\end{corollary}

\begin{proof}
	Using the same notation as in Theorem \ref{9.3} we have that $\gamma = \GCD(\theta,\beta)$ is either a constant or a singular inner function. So either $\ker D_{\varphi}^{\theta}=\{0\}$ and in that case the Fredholmness of  $D_{\varphi}^{\theta}$ is equivalent to its invertibility by Corollary \ref{6.5}, or  $\ker D_{\varphi}^{\theta}$ is infinite dimensional and in this case  $D_{\varphi}^{\theta}$ is neither Fredholm nor invertible.
\end{proof}

\begin{corollary}\label{9.6}
	Let $\beta$ be an inner function and $\gamma = \GCD(\theta,\beta)$. Then $\ker D_{\bar\beta}^{\theta} = \beta \frac{\theta}{\gamma}K_{\gamma}$. The operator $D_{\bar\beta}^{\theta} $ is Fredholm if and only if $\gamma$ is a finite Blaschke product, and it is invertible if and only if $\begin{bmatrix}
	 	\beta\\
	 	\theta
	\end{bmatrix}\in CP^+$.
\end{corollary}
It follows from Corollary \ref{9.6} that we have
\[\ker D_{\bar\theta}^{\theta}=\theta K_{\theta}\] and $\theta \in \ker D_{\bar\theta}^{\theta}$ if and only if $\theta(0)=0$ (cf. \cite[Example 2.5]{Ding Sang}).

If $\varphi \in \mathcal{G}L^{\infty}$ and $\varphi^{-1}\in\overline{H^{\infty}}$, then we can reduce the study of Fredholmness and invertibility of $D_{\varphi}^{\theta}$ to that of $(D_{\varphi}^{\theta})^*=D_{\bar\varphi}^{\theta}$ and use Theorem \ref{9.3}. Regarding the kernel, we have the following:

\begin{corollary}\label{9.7}
	Let $\varphi \in \mathcal{G}L^{\infty}$. Assume that $\varphi^{-1}\in\overline{ H^{\infty}}$ and let $\overline{\varphi^{-1}}=\beta a_+$ be its standard inner--outer decomposition ($\beta$--inner, $a_+$--outer). Denote $\gamma = \GCD(\theta,\beta)$. Then
		\[\ker D_{\varphi}^{\theta}=C_{\theta}(\ker(D_{\varphi}^{\theta})^*)=
	\overline{\left(\frac{\beta}{\gamma}\right)}\overline{a_+}\bar z \overline{K_{\gamma}}\subset \overline{\left(\frac{\beta}{\gamma}\right)}H^2_-.\]
\end{corollary}

\subsection{Rational symbols}

Let now $\varphi$ be continuous on the unit circle $\mathbb{T}$, $\varphi \in C(\mathbb{T})$. In this case we can describe the essential spectrum of $D_{\varphi}^{\theta}$ (see also \cite[Theorem 4.3]{Ding Qin Sang}).

For an inner function $\theta$,
$$\Sigma(\theta)=\{w\in\clos\mathbb{D}:\ \liminf_{z\to w}|\theta(z)|=0\}.$$

\begin{theorem}\label{9.8}
	If $\varphi \in C(\mathbb{T})$, then $\sigma_e(D_{\varphi}^{\theta})=\varphi(\mathbb{T})$.
\end{theorem}

\begin{proof}
	By \eqref{w4.3}, we have $\varphi(\mathbb{T})\subset \sigma_e(D_{\varphi}^{\theta})$. Conversely, if $\lambda\in \mathbb{C}\setminus \varphi(\mathbb{T})$, then $\varphi-\lambda\in \mathcal{G}L^{\infty}$ and $D_{\varphi}^{\theta}$ is Fredholm if and only if $A_{\frac 1{\varphi-\lambda}}^{\theta}$ is Fredholm. By Theorem 5.3 in \cite{CP_spectral} we have that this happens if and only if $0\notin g(\Sigma(\theta)\cap\mathbb{T})$ with $g=\frac 1{\varphi-\lambda}$, which is always true.
\end{proof}


\begin{corollary}\label{9.9}
	If $\varphi \in C(\mathbb{T})$, then $D_{\varphi}^{\theta}$ is Fredholm if and only if $\varphi$ is invertible in $C(\mathbb{T})$.
\end{corollary}

We can obtain further results if $\varphi$ is rational and continuous on $\mathbb{T}$, i.e., $\varphi \in \mathcal{R}$. We start by describing the kernels of dual truncated Toeplitz operators with rational symbols.

\begin{theorem}\label{9.10}
	Let $R\in \mathcal{R}$ and $R=P/Q$, where $P$ and $Q$ are polynomials without common zeros. Then $f\in \ker D_R^{\theta}$ if and only if there  is a decomposition $f=f_-+\theta \tilde f_+$, $f_-\in H_-^2,\  \tilde f_+\in H^2$ such that there are polynomials $P_1$, $P_2$ with $\deg (P_2) \leqslant \max\{\deg (P),\deg(Q)\}-1$ such that
	\begin{equation}\label{o9.1}
	f_-=\tfrac{P_1}P,\quad \tilde f_+=\tfrac{P_2}P,\quad \tfrac{P_1}Q+\theta \tfrac{P_2}Q\in K_{\theta}.
	\end{equation}%
\end{theorem}%
\begin{proof}
	Let $f=f_-+\theta \tilde f_+$ with $f_-\in H_-^2$, $\tilde f_+\in H^2$. Then $f\in \ker  D_R^{\theta}$ if and only if $R(f_-+\theta \tilde f_+)\in K_{\theta}$. The latter happens if and only if there exist $k_+\in H^2$ and $k_-\in H^2_-$ such that
	$$
	\left \{ \begin{array}{c}
	\frac PQ f_-+\frac PQ \theta \tilde f_+=k_+,\\
	\bar\theta k_+=k_-,
	\end{array} \right.$$
	or equivalently,
\begin{equation}\label{o9.4}
	\left \{ \begin{array}{c}
	P f_-=k_+Q-P\theta\tilde f_+,\\
	P\tilde f_+=k_-Q-\bar\theta Pf_-.
	\end{array} \right.
	\end{equation}
	
	By a generalization of Liouville's theorem, both sides of the first equation in \eqref{o9.4} must be equal to a polynomial $P_1$ such that $\frac{P_1}P \in H_-^2$ and, analogously, both sides of the second equation in \eqref{o9.4} must be equal to a polynomial $P_2$ such that $\frac{P_2}P \in H^2$ and $P_2=Qk_--P\bar\theta f_-$. So the degree of $P_2$ is appropriate and $$k_+=\tfrac{P_1}Q+\theta \tfrac PQ\tilde f_+=\tfrac{P_1}Q+\theta \tfrac{P_2}Q \in K_{\theta} .$$
	
	Conversely, if $P_1$ and $P_2$ are polynomials satisfying desired conditions, then for $f_-=\frac{P_1}P$, $\tilde f_+=\frac{P_2}P$ we have that $f_-\in H_-^2$, $\tilde f_+\in H^2$ and $$R(f_-+\theta \tilde f_+)=\tfrac PQ\left(\tfrac{P_1}P+\theta\tfrac{P_2}P\right)=\tfrac{P_1}Q+\theta\tfrac{P_2}Q\in K_{\theta},$$ so that $f=f_-+\theta \tilde f_+\in \ker D_R^{\theta}$.
\end{proof}

The previous theorem enables us to characterize the points $\lambda \in \sigma(D_R^{\theta})$ for $R\in \mathcal{R}$, as follows.

\begin{theorem}\label{9.11}
	If $R\in \mathcal{R}$, then $$\sigma(D_R^{\theta})=R(\mathbb{T}) \cup \sigma_p(D_R^{\theta})=\sigma_e(D_R^{\theta})\cup \sigma_p(D_R^{\theta}).$$
\end{theorem}

\begin{proof}
	From Theorem \ref{9.8} we have $\sigma_e(D_R^{\theta})=R(\mathbb{T}) \subset \sigma(D_R^{\theta})$. If $\lambda \notin R(\mathbb{T})$, then $D_{R-\lambda}^{\theta}$ is Fredholm, so it is invertible if and only if it is injective by Corollary \ref{6.5}. Therefore $\lambda\in \sigma(D_R^{\theta})$ if and only if $\ker D_{\varphi-\lambda}^{\theta}\neq \{0\}$, i.e., $\lambda$ is an eigenvalue of $D_{\varphi}^{\theta}$.
\end{proof}

As an application we study the spectra of the dual truncated shift $D_{z}^{\theta}$ and its adjoint -- the dual truncated backward shift $D_{\bar z}^{\theta}$. We start by studying their kernels. The next result is a consequence of Theorems \ref{9.10} and \ref{9.11}.

\begin{theorem}\label{9.12}If $\lambda\in \mathbb{T} \cup \mathbb{D^-}$ or if $\lambda \in \mathbb{D}$ and $\theta(0)\neq 0$
	then $\ker D_{z-\lambda}^{\theta}=\ker D_{\bar z-\bar\lambda}^{\theta}=\{0\}$. If $\lambda \in \mathbb{D}$, $\theta(0)=0$, then
	\[\ker D_{z-\lambda}^{\theta}=\sspan\left\{\tfrac 1{z-\lambda}\right\}\subset H_-^2\]
	and
	\[\ker D_{\bar z-\bar\lambda}^{\theta}=\sspan\left\{\tfrac{\theta}{1-\lambda z}\right\}\subset \theta H^2.\]
\end{theorem}
\begin{proof}
	By Theorem \ref{9.10}, $f\in \ker D_{z-\lambda}^{\theta}$ if and only if there are constants $A,B\in\mathbb{C}$ such that $f=f_-+\theta\tilde f_+$ with $f_-=\frac A{z-\lambda}$ and $\tilde f_+=\frac B{z-\lambda}$. Note that if $\lambda \in \mathbb{T}\cup \mathbb{D}^-$, then $A=0$, and if $\lambda \in \mathbb{T}\cup \mathbb{D}$, then $B=0$. Therefore
	\begin{enumerate}
	\item if $\lambda \in \mathbb{T}$, then $\ker D_{z-\lambda}^{\theta}=\{0\}$;
	\item if $\lambda \in \mathbb{D}^-$, then $\ker D_{z-\lambda}^{\theta}=\{\theta \tfrac B{z-\lambda} : B\in \C,\ \theta B\in K_{\theta}\}$;
	\item if $\lambda \in \mathbb{D}$, then $\ker D_{z-\lambda}^{\theta}=\{\tfrac A{z-\lambda} : A\in \C\cap  K_{\theta}\}=\{0\}$.
\end{enumerate}
 So $\ker D_{z-\lambda}^{\theta}=\{0\}$ unless $\lambda \in \mathbb{D}$ and $\theta(0)=0$. By Theorem \ref{6.3}, we have $\ker D_{\bar z-\bar\lambda}^{\theta}=C_{\theta} (\ker D_{z-\lambda}^{\theta})$.
\end{proof}

\begin{corollary}\label{9.13}
	\begin{enumerate}
		\item $\sigma_e(D_z^{\theta})=\sigma_e(D_{\bar z}^{\theta})=\mathbb{T}$;
		\item $\sigma(D_z^{\theta})$ and $\sigma(D_{\bar z}^{\theta})$ are the disjoint union of $\mathbb{T}$ 
		 with $\sigma_p(D_z^{\theta})=\sigma_p(D_{\bar z}^{\theta})$, where $\sigma_p(D_z^{\theta})=\emptyset$ if $\theta(0)\neq 0$ and $\sigma_p(D_z^{\theta})=\mathbb{D}$ if $\theta(0)=0$.
	\end{enumerate}
\end{corollary}

\end{document}